\documentclass[10pt]{amsart}
 \usepackage{amssymb,amsmath}

\usepackage{graphicx}
\usepackage{amsmath}
\usepackage{amsfonts}
\usepackage[english]{babel}
\usepackage{amsthm}
\usepackage{latexsym}

\numberwithin{equation}{section}
\newtheorem{theorem}{Theorem}[section]
\newtheorem{corollary}[theorem]{Corollary}

\newtheorem{lemma}[theorem]{Lemma}
\newtheorem{proposition}[theorem]{Proposition}
\newtheorem{remark}[theorem]{Remark}

\def\neweq#1{\begin{equation}\label{#1}}
\def\endeq{\end{equation}}

\newcommand{\R}{\mathbb{R}}

\newcommand{\om}{\Omega}

\newcommand{\eps}{\epsilon}

\newcommand{\beq}{\begin{equation}}
\newcommand{\eeq}{\end{equation}}
\newcommand{\beqs}{\begin{equation*}}
\newcommand{\eeqs}{\end{equation*}}
\newcommand{\beqn}{\begin{eqnarray}}
\newcommand{\eeqn}{\end{eqnarray}}
\newcommand{\beqns}{\begin{eqnarray*}}
\newcommand{\eeqns}{\end{eqnarray*}}

\newcommand{\n }{\nabla }

\renewcommand{\O}{\Omega}
\newcommand{\po}{\partial\Omega_{\lambda}}
\renewcommand{\(}{\left(}
\renewcommand{\)}{\right)}

\newcommand{\lm}{\lambda}

\begin{document}

\title[Concentration along geodesics for a nonlinear Steklov problem]{Concentration along geodesics for a nonlinear Steklov problem arising in corrosion modelling}

\author{Carlo D. Pagani}
\address{Carlo D. Pagani, Istituto Lombardo, via Borgonuovo 25, 20121 Milano, Italy  }
\email{carlo.pagani@polimi.it}

\author{Dario Pierotti}
\address{Dario Pierotti, Dipartimento di Matematica, Politecnico di Milano, Piazza Leonardo da Vinci 32, 20133 Milano, Italy  }
\email{dario.pierotti@polimi.it}

\author{Angela Pistoia}
\address{Angela Pistoia, Dipartimento di Scienze di Base e Applicate per l'Ingegneria,  Sapienza Universit\`a di Roma,
via Antonio Scarpa 16, 00161 Roma, Italy}
\email{angela.pistoia@uniroma1.it}

\author{Giusi Vaira}
\address{Giusi Vaira,  Dipartimento di Scienze di Base e Applicate per l'Ingegneria,  Sapienza Universit\`a di Roma,
via Antonio Scarpa 16, 00161 Roma, Italy}
\email{vaira.giusi@gmail.com}


\begin{abstract}
We consider the problem of finding pairs $(\lambda, \mathfrak{u})$, with $\lambda>0$ and  $\mathfrak{u}$ a harmonic function in a three dimensional  torus-like domain $\mathcal{D}$, satisfying the nonlinear boundary condition $\partial_{\nu}\mathfrak{u}=\lambda\,\sinh \mathfrak{u}$ on $\partial \mathcal{D}$. This type of boundary condition arises in corrosion modelling (Butler Volmer condition). We prove existence of solutions which concentrate along some geodesics of the boundary   $\partial \mathcal{D}$ as the parameter $\lambda$ goes to zero.
\end{abstract}

\date{}
\maketitle

\section{Statement of the problem and main results}

In mathematical modeling of electrochemical corrosion a suitably defined galvanic potential satisfies an elliptic equation  (namely, the Laplace or Poisson equation in simplest cases) in a given domain $\mathcal{D}$, whose boundary is partly electrochemically active and partly inert. In the inactive boundary region,
the current density flow is of course zero, but in the active part it is modeled (by interpolating experimental data) by a difference of two exponentials according to the so-called Butler-Volmer formula (see \cite{vogexiu} for a detailed discussion of the model). Then the resulting mathematical problem consists in solving an elliptic equation complemented with a boundary condition of Neumann type, namely

\begin{equation}
\label{voge}
\partial_{n}\mathfrak{u}(y) = \lambda\,\mu(y) \big [e^{\alpha \mathfrak{u}(y)}-e^{-(1-\alpha)\mathfrak{u}(y)}\big ],\quad\quad y\in\partial \mathcal{D}
\end{equation}
Here $n$ is the outward unit normal to $\partial \mathcal{D}$, $\mathfrak{u}$ is the surface potential, $\alpha\in (0,1)$ is a constant depending on the constituents of the electrochemical system, $\mu(y)$ is a non negative bounded function which distinguishes between the active and the inert boundary regions (typically $\mu(y)$ is the characteristic function of some subset of $\partial\mathcal{D}$) and $\lambda$ is a positive parameter.

Due to the exponential grow of the nonlinear boundary term, this problem has beeen studied (usually for Laplace or Poisson equations) in two dimensions:  see  \cite{vogexiu}, \cite{vogekav}, \cite{PP1}, \cite{PP2}.

In the physically relevant three dimensional case, little is known about existence and properties of solutions (see \cite{PP3}, \cite{ppvz}).

A case  which presents some interest for applications arises when the body $\mathcal{D}$ is a three dimensional annular shaped domain, namely it can be represented in the form
\begin{equation}
\label{dom}
\mathcal{D} =\Big\{(y_1,y_2,y_3)\,\in\,\R^3\Big |\, \big (\sqrt{y_1^2+y_2^2},y_3\big )\in\om \Big\}
\end{equation}
and $\om$ is a smooth bounded domain in $\R^2$ such that
\begin{equation}
\label{dom2d}
\overline\om \subset \big\{(x_1,x_2)\in\R^2\,\big |\, x_1>0 \big\}
\end{equation}
Clearly, the domain $\mathcal{D}$ is $\mathcal{G}$-invariant for the action of the group $\mathcal{G}$ given by
\begin{equation}
\nonumber
g(y_1,y_2,y_3)=(\tilde g(y_1,y_2), y_3)
\end{equation}
where $\tilde g\in \mathcal{O}(2)$
(the group of linear isometries of $\R^2$).

The chosen geometry aims at modellizing the corrosion of torus-like bodies. Then, we consider the boundary value problem
\begin{equation}
\left\{\begin{aligned}
&\Delta \mathfrak{u} = 0\quad {\rm in}\,\, \mathcal{D}\\
 & \partial_{n}\mathfrak{u}= \lambda \,\sinh \mathfrak{u}\quad {\rm on}\,\, \partial \mathcal{D}
\end{aligned}\right.
\label{nlbvp}
\end{equation}
Note that we have chosen in \eqref{voge} $\alpha=1/2$, $\mu\equiv 1$, and we wrote $\mathfrak{u}$ instead of $\mathfrak{u}/2$ for the \emph{harmonic} potential.

In such situation
it is natural to look for solutions which are $\mathcal{G}$-invariant, i.e. they are \emph{axially symmetric functions} of the form
\begin{equation}
\label{solsimm}
\mathfrak{u}(y_1,y_2,y_3)=u(x_1,x_2 )
\end{equation}
where $x_1=(y_1^2+y_2^2)^{1/2}$ and $x_2=y_3$.

Now an easy computation shows that \eqref{nlbvp} is equivalent to the following problem for $u$ :
\begin{equation}
\left\{\begin{aligned}
& \mathrm{div}\big (x_1\nabla u\big ) =0\quad {\rm in}\,\, \om\\
  &\partial_{\nu}u=\lambda \,\sinh u\quad {\rm on}\,\, \partial \om
\end{aligned}\right.\label{nl2dpb}
\end{equation}

Thus, we are led to study the more general anisotropic two-dimensional problem
\begin{equation}
\left\{\begin{aligned}
&\Delta_a u  = 0\quad {\rm in}\,\, \om\\
  &\partial_{\nu}u= \lambda \,\sinh u\quad {\rm on}\,\, \partial \om
\end{aligned}\right.\label{gen2dpb}
\end{equation}
with $a\in L^{\infty}(\om)$ such that
\begin{equation}
\label{conda}
0<a_0\le a(x_1,x_2)\le a_1<\infty\quad\quad \mathrm{a.e.}\,\,\mathrm{in}\,\,\om\,,
\end{equation}
$\lambda$ is a positive parameter and $\Delta_a u$ is defined as

\begin{equation}
\label{defa}
\Delta_a u=\frac{1}{a(x)}\mathrm{div}\big (a(x)\nabla u\big )=\Delta u+\nabla\log a(x)\cdot\nabla u
\end{equation}
for every $u\in H^1(\om)$.

We remark that problem \eqref{nl2dpb} corresponds to choosing $a(x_1,x_2)\equiv x_1$. \\

The main goal of the present paper is to extend to the anisotropic problem \eqref{gen2dpb} the principal results obtained for
$a\equiv 1$ both concerning existence of multiple solutions and the limiting behaviour of the solutions for $\lambda\to 0^+$.

\smallskip
In what follows we first (section $2$) prove existence of global multiple solutions of problem \eqref{gen2dpb} and, as a consequence, of problem \eqref{nlbvp}.
This is done by a mild modification of the variational approach used in \cite{PP2}. We prove the following result:

\begin{theorem}
\label{mainth} Let $\om\subset\R^2$ be a Lipschitz domain and let $a\in L^{\infty}(\om)$ such that
\eqref{conda} is satisfied.
Then for every $\lambda>0$  there exist infinitely many solutions (in $H^1(\om)$) of problem \eqref{gen2dpb}.
\end{theorem}
Then, a simple Corollary states the existence of solutions of problem \eqref{nlbvp}:

\begin{corollary}
\label{coro1}
For every $\lambda>0$ there are infinitely many axially symmetric solutions to problem \eqref{nlbvp}.
\end{corollary}
Moreover, in subsection \ref{23} a first approach to the description of the behaviour of the solutions for $\lambda\to 0^+$ is considered. For $a\equiv 1$ this study was done in \cite{vogekav}; there the authors prove, for a bounded $\mathcal{C}^{2,\alpha}$ domain $\om\subset \R^2$, that all solutions have energies that are of order $\log(1/\lambda)$; the limit of the boundary flux along subsequences is a sum of Dirac masses located at a finite set of points, potentially accompanied by a regular part of definite sign. These results are a generalization of those obtained in \cite{BV} by studying the behaviour of explicit solutions of the problem in the unit disk. We show that the same kind of results hold for the anisotropic case, and we sketch the main arguments of the proofs.

In \cite{davila}, considering again the case $a\equiv 1$,  the authors prove that in any domain $\om$ there are at least two distinct families of solutions which exhibit exactly the qualitative behaviour of the explicit solutions found in \cite{BV}, namely with limiting boundary flux given by an array of Dirac masses with weight $2\pi$ and alternate signs. We will show in section \ref{3}, that, for the anisotropic problem \eqref{gen2dpb}, the situation is substantially the same as depicted in \cite{davila}. In that case the arguments used to extend the results of \cite{davila} to the anisotropic case are in some part different from those given in \cite{davila}.

\smallskip
In order to state our result let $x=(x_1,x_2)\in\Omega$, $y=(y_1,y_2)\in\partial\Omega$   and let $G_a(x, y)$ be the Green's function for the Neumann problem

\beq
\label{Greena}
\left\{
\begin{array}{lr}
-(\Delta_a)_x G_a(x, y)=0\qquad\mbox{in}\,\,\om\\\\
\displaystyle\frac{\partial G_a}{\partial \nu_x}=2\pi\delta_y(x)-\frac{2\pi}{\int_{\partial\om}a(x)}a(y)\qquad\quad\mbox{on}
\,\,\, \partial\om\\\\
\displaystyle\int_{\partial\om}G_a(x, y)=0
\end{array}
\right.
\eeq
and define $H_a(x, y)$ as the regular part of $G_a(x, y)$, namely
\beq\label{robina}
H_a(x, y)=G_a(x, y)-\log\frac{1}{|x-y|^2}.
\eeq

We say that $\xi\in \partial\Omega$ is a $C^1-$stable critical point of  $a |_{\partial\om}$
 if $\nabla_{{\partial\Omega}}a(\xi) =0$ and  the  local Brouwer degree
 $\mathrm{deg}\left(\nabla_{{\partial\Omega}}a ,  B(\xi,\rho) \cap \partial\Omega   ,0\right)$ is well defined (for $\rho$   small enough) and it is not zero.
It is easy to see that any  strict local minimum   point
  or   strict local maximum point or non degenerate critical point of $a |_{\partial\om}$   is a $C^1-$stable critical point.
\\

Now, the main result is the following
\begin{theorem}\label{principale}
Let $\om\subset\mathbb R^2$ be a smooth domain and $a\in C^1(\overline\Omega)$. Let  $\xi_1, \xi_2\in\partial\om$ be two different $C^1-$stable critical points of $a |_{\partial\om}$. Then there is $\lambda_0>0$ such that, for $0<\lambda<\lambda_0$, there is a sign-changing solution $u_\lambda $ of the problem \eqref{gen2dpb} of the form
\beq\label{solth}
u_\lambda(x)=\log\frac{2\mu_1}{|x-(\xi_1 +\lambda\mu_1\nu_1)|^2}-\log\frac{2\mu_2}{|x-(\xi_2 +\lambda\mu_2\nu_2)|^2}+O(1)
\eeq
where $\nu_1$ and $\nu_2$ denote the unit outer normals to $\partial\om$ at the point $\xi_1$ and $\xi_2$ respectively, and the parameters $\mu_1$ and $\mu_2$ are explicitly given by
$$ \mu_1= {1\over2}e^{H_a(\xi_1, \xi_1)-G_a(\xi_1, \xi_2)}, \ \mu_2= {1\over2}e^{H_a(\xi_2, \xi_2)-G_a(\xi_2, \xi_1)}.$$
In particular, the solution $u_\lambda$ concentrates positively and negatively at the points $\xi_1$ and $\xi_2$ respectively, as $\lambda$ goes to zero.
\end{theorem}

According to the previous discussion, the corresponding result for problem \eqref{nlbvp} reads as follows:
\begin{corollary}\label{principalegen}
Let $\om\subset\mathbb R^2$ be a smooth domain and $a\in C^1(\overline\Omega)$. Let  $\xi_1, \xi_2\in\partial\om$ be two different $C^1-$stable critical points  of $a (x)=x_1$ restricted on ${\partial\om}$. Then there exists $\lambda_0>0$ such that for any $\lambda\in(0, \lambda_0)$, problem \eqref{nlbvp} has a sign-changing axially symmetric solution ${\mathfrak u}_\lambda$ which concentrates positively and negatively along two geodesics  of the boundary $\partial\mathcal D$ which are the $\mathcal G-$orbits of $\xi_1 $ and
$\xi_2$ respectively, as $\lambda $ goes to zero.
\end{corollary}

 The proof of Theorem \ref{principale} relies
  on a very well-known finite dimensional procedure and it is carried out in Section \ref{3}. We shall omit many details of this proof because they can be found up to minor modifications in \cite{davila}, where problem \eqref{gen2dpb} has been studied with $a(x)\equiv1.$ We only compute what cannot be deduced from known results. It   must be mentioned that in our case the anisotropic term $a$ allows to solve the reduced problem by finding the location of the concentration points in
terms of critical points of $a$ itself.

Our result suggests  that  it should be possible to find  solutions which concentrate along two geodesics of the boundary of a more general torus-like 3-dimensional domain which is not necessarily axially symmetric.

\section{Existence and multiplicity of solutions}
\label{2}
The content of this section almost repeats the arguments developed in \cite{PP2}: recall that we have chosen the weight function $\mu(x)$ appearing in \eqref{voge} identically $1$, and that we deal with the operator $\Delta_a$ instead of the Laplacean; therefore many statemets are simple rephrasements of the corresponding statements given in \cite{PP2} and are not given here in full details.
Let us first recall that the approach to the non linear problem relies on the solution
of a related linear Steklov eigenvalue problem on the
boundary. We summarize here (without proofs) the crucial
results about this problem.\\
Let $\om\subset \R^N$ be a bounded Lipschitz domain and consider the following
linear Steklov eigenvalue problem in $H^1(\om)$:
\begin{equation}
\left\{\begin{aligned}
& \mathrm{div}\big (a(x)\nabla u\big )= 0\quad {\rm in}\,\, \om\\
 & \partial_{\nu}u= \lambda \, u\quad {\rm on}\,\, \partial \om
\end{aligned}\right.\label{linpb}
\end{equation}
It is easily seen that for $\lambda>0$ the solutions to \eqref{linpb} belong to the subspace
$H^1_{a}\subset H^1(\om)$ defined as follows:

\begin{equation}
\label{sub1}
H^1_{a}\equiv \Big \{u\in H^1(\om),\quad \int_{\partial\om}a\,u = 0\Big\}.
\end{equation}
It can be shown (by a classical \textit{reductio ad absurdum} argument, see e.g. \cite{PP1}) that in $H^1_{a}$  the Dirichlet norm $\int_{\om}|\nabla u|^2$ is
equivalent to the $H^1$ norm and that \eqref{linpb} is
equivalent to the following variational problem:

\noindent Find $u\in H^1_{a}$, $u\neq 0$, such that
\begin{equation}
\label{eqvar}
\int_{\om}a\,\nabla u\,\nabla \varphi=\lambda\int_{\partial\om}a\,u\varphi
\end{equation}
holds for every $\varphi\in H^1$.
Furthermore, the expression
\begin{equation}
\label{eqnorm}
\|u\|_a^2=\int_{\om}a\,|\nabla u|^2+\Bigl (\int_{\partial\om}a\,u\Bigr
)^2
\end{equation}
defines an equivalent norm in $H^1(\om)$. We will consider the scalar product in
$H^1(\om)$ associated to this equivalent norm; by a slight adaptation of the proof in \cite{PP1}, theorem $2.1$, we have the following
result :

\begin{proposition}
\label{mt} Problem \eqref{eqvar} has infinitely many
eigenvalues
$$ 0<\lambda_1\le\lambda_2\le ...$$
each of finite multiplicity and such that
$|\lambda_n|\to +\infty$.
The eigenvalue $\lambda_0=0$ corresponds to the constant solutions of the
homogeneous Neumann problem. Moreover, we can take all the eigenfunctions $v_n$, $n\ge 0$  orthogonal and normalized with
respect to the scalar product associated  to the equivalent norm \eqref{eqnorm} and the following decomposition holds:
\begin{equation}
\label{orthdec}
H^1=H^1_0\oplus V_{a}
\end{equation}
where the subspace $V_{a}$ is spanned by the eigenfunctions $v_n$ (note that $V_{a}$ contains the subspace of the constant functions).
\end{proposition}

\begin{remark}
\label{reg}
Global regularity of the eigenfunctions depends on the
weight $a(x)$ and on the regularity of the boundary $\partial\om$.
It can be shown that on a  Lipschitz domain $\om$
the eigenfunctions are bounded and continuous on
$\overline\om$.
\end{remark}

\subsection{Main estimates}
\label{21}

Let us consider now the non linear problem \eqref{gen2dpb};
we define the \emph{even functional}
\begin{equation}
\label{funz}
E_{\lambda}(u)=\frac{1}{2}\int_{\om}a\,|\nabla u|^2-\lambda\int_{\partial\om}a\, \big (\cosh u-1\big )
\end{equation}
where $u\in H^1(\om)$.
Note that
\begin{equation}
\label{deco}
\cosh u-1=\frac{u^2}{2}+u^4 h(u)
\end{equation}
where, for every $q>1$,
$$ h:\,H^1(\om)\rightarrow L^q(\partial\om)$$ is \emph{bounded} (see \cite{vogekav}, lemma 2.1).

Define further
\begin{equation}
\label{defS} S_R = \{u\in H^1:\,\,\|u\|_a=R\};
\end{equation}
where $\|\,\|_a$ is the equivalent norm defined in \eqref{eqnorm}.
Then, we have
\begin{lemma}
\label{link1} For every $\lambda>0$, there exists
$R>0$ and a closed subspace $V^+\subseteq H^1(\om)$ with \allowbreak codim\nobreak \,$V^+\nobreak < \nobreak\infty$
such that
$$E_{\lambda}(u)\ge c_0>0,$$ for every $u\in S_R\cap V^+$.
\end{lemma}

\begin{proof}
By \eqref{deco} we have
\begin{equation}
\label{funz1}
E_{\lambda}(u)\ge \frac{1}{2}\int_{\om}a\,|\nabla u|^2-\frac{\lambda}{2}\int_{\partial\om}a\, u^2
-{\lambda}\int_{\partial\om}a\,|u|^{4}h(u),
\end{equation}
and for every $q>1$, $p=q/(q-1)$, the integral in the last term can be bounded as follows
\begin{equation}
\label{est1}
\Big|\int_{\partial\om}a\,|u|^{4}h(u)\Big |
\le \|a\|_{L^{\infty}(\partial\om)}\|h(u)\|_{L^{q}(\partial\om)}
\Bigl (\int_{\partial\om}|u|^{4p}\Bigr )^{1/p}
\le C\,\|u\|_a^{4}=C\,R^{4},
\end{equation}
Let us now consider the quadratic part of the functional. Since the sequence
of positive eigenvalues $\lambda_n$ of the linear problem is
unbounded, there exist non negative integers $k$, $r$, such that
\begin{equation}
\label{ass3}
\lambda_k\le\lambda<\lambda_{k+r}
\end{equation}
Then, we set
\begin{equation}
\label{v2} V^+ =H^1_0\oplus{\rm span}_{n\ge k+r}\,\{v_n\},
\end{equation}
For every $u\in S_R\cap V^+$ we get
\begin{equation}
\label{est2}
\frac{1}{2}\int_{\om}a\,|\nabla u|^2-\frac{\lambda}{2}\int_{\partial\om}a\, u^2
\ge \frac{1}{2}\Bigl (1-\frac{\lambda}{\lambda_{k+r}}  \Bigr )\int_{\om}a\,|\nabla
 u|^2=\frac{1}{2}\Bigl (1-\frac{\lambda}{\lambda_{k+r}}  \Bigr )R^2
\end{equation}
Then, the lemma follows by taking $R$ small enough.
\end{proof}
We are now going to construct closed, \emph{finite dimensional} subspaces,
$V^-\subset H^1(\om)$ such that:
\begin{itemize}
  \item dim $V^->$ codim $V^+$;
  \item $E_{\lambda}(u)\le c_{\infty}<\infty$ for every $u\in V^-$
\end{itemize}

To this aim, let $v_{n_i}$, $1\le i\le l$ be \emph{any} finite sequence
of $l$ eigenfunctions, with $l>$ codim $V^+$, corresponding to the eigenvalues
\begin{equation}
\label{leigenv}
\lambda_{n_1}\le\lambda_{n_2}\le...\le\lambda_{n_l}.
\end{equation}
Let us define:
\begin{equation}
\label{v1-}
V^-={\rm span}_{1\le i\le
l}\,\{v_{n_i}\},
\end{equation}
The next lemma provides the key estimates at infinity on the functional
\eqref{funz}.
\begin{lemma}
\label{link2} Let  $V^-$ be defined by
\eqref{v1-}. Then, for every $\lambda>0$ we have $E_{\lambda}(u)<0$ for any $u\in V^-$ with
large enough norm. As a consequence, there exists $c_{\infty}<\infty$ such that
$$E_{\lambda}(u)\le c_{\infty}\quad\quad\forall \,u\in V^-$$
\end{lemma}
\begin{proof}
We first assume $\lambda_{n_1}>0$. For notational
simplicity, from now on we set $v_{n_i}=v_i$,
$\lambda_{n_i}=\lambda_i>0$ ($1\le i\le l$).
Thus, we can write any $u\in V^-$ in the form
$$u=\sum_{i=1}^l t_i v_i$$
Recalling remark \ref{reg}, $u$ is a bounded continuous function on
$\om$, so that $\sinh u\in H^1(\om)$ and  by proposition
\ref{mt} the variational equation
\begin{equation}
\label{eqvaruj}
\int_{\om}a\,\nabla v_j\,\nabla
\sinh u=\lambda_j\int_{\partial\om}a\,v_j\, \sinh u
\end{equation}
holds for every $j$, $1\le j\le l$.
Multiplying by $t_j$ and summing up from $j=1$ to $l$, we find
\begin{equation}
\label{eqvarsum}
\int_{\om}a\, |\nabla u |^2 \cosh(u)=\sum_{j=1}^l\lambda_j t_j\int_{\partial\om}a\,v_j\,\sinh u
\end{equation}
On the finite dimensional space $V-$ the functional \eqref{funz} takes the form
$$E_{\lambda}(u)\equiv f(t_1,...,t_l)=\sum_{i,j=1}^l \frac{t_it_j}{2}\int_{\om}a\,\nabla v_i\cdot\nabla v_j
-\lambda\int_{\partial\om}a\,\big ( \cosh u-1\big )$$
\begin{equation}
\label{ftk}
=\frac{1}{2}\sum_{i=1}^l t_i^2-\lambda\int_{\partial\om}a\,\big ( \cosh u-1\big ),
\end{equation}
where we used orthogonality and normalization of $v_i$ with respect to the
inner product defined by the equivalent norm \eqref{eqnorm}.
Then, we have
$$\Bigl (\sum_{j=1}^l\lambda_j t_j\partial_{t_j}
\Bigr)f(t_1,...,t_l)=\sum_{j=1}^l \lambda_j t_j^2
-\lambda \sum_{j=1}^l\lambda_j t_j\int_{\partial\om}a\,v_j\, \sinh u$$
\begin{equation}
\label{derdirf}
\le
\lambda_l\sum_{j=1}^l  t_j^2-\lambda\int_{\om}a\, |\nabla u |^2\cosh u=
\int_{\om} a\, |\nabla u |^2\Bigl [{\lambda_l}-\lambda\,\cosh u\Bigr ],
\end{equation}
Now, it can be proved (see \cite{PP1}, lemma 3.5) that the last term is strictly negative for $\|u\|_a=\sqrt{t_1^2+...t_l^2}$ large enough.
But the first term is the
derivative of the function $f$ along the curves
$$t_1=c_1 e^{\lambda_1 s},...,t_l=c_l e^{\lambda_l
s}\quad\quad s\in \R$$
(orthogonal to the hypersurfaces $\lambda_1t^2_1+...+\lambda_lt_l^2=$ constant); hence, for large $\sqrt{t_1^2+...t_l^2}$, the function $f$ is \emph{strictly
decreasing} along these curves.

We conclude that $f(u)<0$ for $u\in V^-$ with $\|u\|_a$ large enough; since $f$
is continuous and $V^-$ has finite dimension, we have
$$\sup_{u\in V^-}f(u)=c_{\infty}<\infty$$
We are left to show that  we may allow
$\lambda_{n_1}=0$ in \eqref{leigenv}. This can be proved by the same arguments as in the proof of lemma $3.5$ of \cite{PP2}.
\end{proof}

\subsection{Proof of theorem \ref{mainth}}
\label{22}

To prove
existence and even multiplicity of solutions to problem \eqref{gen2dpb} for every
$\lambda>0$, we will apply the Symmetric
Mountain Pass Lemma (see \cite{struwe} thm. 6.3); thus, we need to show that the functional
\eqref{funz}
satisfies the Palais-Smale condition; to this aim
the following estimate plays a key role:
\begin{equation}
\label{condps2bb}
    4(\cosh u-1)\le u\sinh u +u^2
\end{equation}

\begin{proposition}
\label{PS2} Let $z_m\in H^1(\om)$ be a sequence such that
$E_{\lambda}(z_m)\rightarrow c$ and $E_{\lambda}'(z_m)\rightarrow 0$ in
$H^1(\om)'$. Then, the sequence $z_m$ is bounded
and the functional \eqref{funz} satisfies the Palais-Smale condition.
\end{proposition}

\begin{proof}
Assume by contradiction (considering a
subsequence if necessary) that $\|z_m\|_a\to +\infty$ and define
$t_m=\|z_m\|_a$, $u_m=t_m^{-1}z_m$. Substituting in the condition
$E_{\lambda}'(z_m)u=o(1)\|u\|_a$, we get

\begin{equation}
\label{succvm}
\int_{\om}a\,\nabla u_m\nabla u-\lambda \int_{\partial\om}a\,  \frac{\sinh(t_m u_m)}{t_m}u=o(1)\|u\|_a/t_m
\end{equation}
Since $u_m$ is bounded in $H^1(\om)$, there is a subsequence (still denoted by
$u_m$) such that $u_m$ converges weakly in $H^1(\om)$
and $u_m|_{\partial\om}$ converges strongly in $L^2(\partial\om)$; it can be proved (see \cite{PP1}, proposition 4.2) that
$u_m|_{\partial\om}\to 0$ a.e.


By choosing $u=u_m$ in \eqref{succvm} we get
\begin{equation}
\label{succtm1}
\int_{\om}a\,|\nabla u_m|^2-\lambda
\int_{\partial\om}a\,  \frac{\sinh(t_m u_m)}{t_m} u_m=o(1/t_m),
\end{equation}
On the other hand, from
$E_{\lambda}(z_m)\rightarrow c$ we also get
\begin{equation}
\label{succH}
\frac{1}{2}\int_{\om}a\,|\nabla u_m|^2-\lambda
\int_{\partial\om}a\, \frac{\cosh(t_m u_m)-1}{t^2_m}={ O}(1/t_m^2)
\end{equation}
By comparison of
\eqref{succtm1} and \eqref{succH} and by taking account of \eqref{condps2bb} we find
$$0=\lambda \int_{\partial\om}a\,
\Bigl[2\frac{\cosh(t_m u_m)-1}{t^2_m}-\frac{\sinh(t_m u_m)}{t_m} u_m\Bigr ]+o(1/t_m)\le$$
$$
\le -2\lambda\int_{\partial\om}a\, \frac{\cosh(t_m u_m)-1}{t^2_m}+\lambda \int_{\partial\om}a\, u_m^2+o(1/t_m)
$$
By recalling that $u_m\to 0$ in $L^2(\partial\om)$ we finally get
\begin{equation}
\label{succomp}
0\le -2\lambda\int_{\partial\om}a\, \frac{\cosh(t_m u_m)-1}{t^2_m}+o(1)
\end{equation}
By the above relation and again by \eqref{succH} we conclude
$$\int_{\om}a\,|\nabla u_m|^2\rightarrow 0$$
so that $\|u_m\|_a\rightarrow 0$,
thus contradicting $\|u_m\|_a=1$.
Then, the norm sequence $\|z_m\|_a$ is bounded and the same holds for $\|z_m\|$. We can write
$$z_m=c_m+\tilde z_m$$ where $c_m$ is a bounded sequence and $\tilde z_m$ is bounded in
$H^1_{a}(\om)$ (see definition \eqref{sub1}).

Now, the linear map $L:H^1_{a}(\om)\rightarrow H^1(\om)'$
$$L(u)\varphi=\int_{\om}a\,\nabla u\nabla \varphi$$ is boundedly invertible (by
the Lax-Milgram theorem) while the operator
$$T(u)\varphi =\int_{\partial\om}a\,\sinh u \varphi$$
maps bounded sets in $H^1(\om)$ to
relatively compact sets in $H^1(\om)'$ (the result follows by an obvious extension of the arguments in
\cite{vogekav}, Lemma 2.1 and Lemma 2.2). By standard results \cite{struwe} Proposition 2.2, it follows that $\tilde z_m$ is relatively compact in
$H^1_{a}(\om)$; then, by the above decomposition, the same holds for $z_m$ in $H^1(\om)$.
\end{proof}

\emph{{Proof of Theorem \ref{mainth}}:}
By
lemmas \ref{link1} and \ref{link2}, for any positive integer $m$ there exist
two closed subspaces $V^+$, $V^-$ of $H^1(\om)$ with dim$V^--$ codim$V^+=m$,
and positive constants $R$, $c_0$, $c_{\infty}$ (the last one depending on $m$)
such that:

$$a) \quad E_{\lambda}(u)\ge c_0\quad \forall u\in V^+,\quad
      \|u\|_a=R;\quad\quad b)\quad E_{\lambda}(u)\le c_{\infty}\quad\forall u\in V^-.$$

Then, by theorem 2.4 of \cite{bbf}, the functional $E_{\lambda}$ possesses at
least $m$ distinct pairs of critical points, corresponding to critical levels
$c_k(\lambda)$, $k=1,2,...,m$, given by
\begin{equation}
\label{minmaxlev}
c_k(\lambda)=\inf_{A\in\Sigma_k}\sup_{u\in A}E_{\lambda}(u)
\end{equation}
where $\Sigma_k$ is the set of closed, symmetric sets $A\subset H^1(\om)$ such that $\gamma^*(A)\ge k$ and
$\gamma^*$ is the \emph{pseudo-index} related to the Krasnoselski genus $\gamma$ and to the subset $S_{R}\cap V^+$ (see
\cite{bbf}, definition $2.8$).

Moreover,
$$c_0\le c_1\le c_2\le...\le c_m\le c_{\infty}.$$
Since this conclusion holds for arbitrary $m$, we get infinitely many critical
points; hence, problem \eqref{gen2dpb} has infinitely many solutions in
$H^1(\om)$. By standard regularity results \cite{cherrier}, if $\om$ is smooth
and $a\in \mathcal{C}^{\infty}(\overline\om)$, we have $u\in
\mathcal{C}^{\infty}(\overline\om)$.

\begin{remark}
In the degenerate case $c_k=...=c_{k+r}=c$ (with $k\ge 1$ and
$k+r\le m$) it was shown in \cite{bbf} that $\gamma(K_c)\ge r+1\ge 2$, where $K_c$
is the set of critical points at level $c$;
since a finite set (not containing the origin) has genus $1$, it follows that
$E_{\lambda}$ has infinitely many critical points at level $c$.
\end{remark}

\begin{remark}
The results of existence and multiplicity obtained in
\cite{vogekav} correspond to the case $a=1$ of Theorem \ref{mainth}  above.
\end{remark}

From the discussion given in the introduction, the existence of solutions of problem \eqref{nlbvp} stated in Corollary \ref{coro1} easily follows.

\subsection{Estimates and limits for the variational solutions}
\label{23}

As discussed in the introduction, we now prove lower and upper bounds for the variational solutions obtained in the previous subsection; more precisely, we show that a branch of solutions corresponding to any of the critical levels $c_k(\lambda)$  blows up in energy (as well as in the Dirichlet seminorm) at the rate $\log(1/\lambda)$ for $\lambda\to 0^+$, while the corresponding normal currents stay bounded in $L^1(\partial\om)$.

We first establish the lower bounds; by observing that a solution $u$ to \eqref{gen2dpb} satisfies
$$0=\int_{\partial\om} a\,\partial_{\nu}u=\lambda\int_{\partial\om} a\,\sinh u$$
we may write $u=u^0+s$, $s\in \R$, where $u^0$ solves the problem
\begin{equation}
\left\{\begin{aligned}
& \Delta_a u^0  = 0\quad {\rm in}\,\, \om\\
&  \partial_{\nu}u^0= \lambda \,\sinh (u^0+s)\quad {\rm on}\,\, \partial \om
\end{aligned}\right.\label{gen2dpb0}
\end{equation}
and
$$\int_{\partial\om} a \, u^0=0,\quad\quad \int_{\partial\om} a\,\sinh (u^0+s)=0$$
By elementary calculations, the second identity yields
\begin{equation}
\label{suzero}
s=s(u^0)\equiv \frac{1}{2}\log\frac{\int_{\partial\om} a\,e^{-u^0}}{\int_{\partial\om} a\,e^{u^0}}
\end{equation}

Now, following \cite{vogekav} and by recalling the bound \eqref{conda}, one first proves the inequality
$$|s(u^0)|\le C_1+C_2\|\nabla u^0\|^2$$
(with positive constants $C_1$, $C_2$ depending only on $\om$ and on $a$) and subsequenly the following \emph{lower estimates} for the solutions to \eqref{gen2dpb}
\begin{equation}
\label{lowest}
E_{\lambda}(u)\ge A\log(\frac{1}{\lambda})-B,\quad\quad \|\nabla u\|^2 \ge A\log(\frac{1}{\lambda})-B
\end{equation}
where $E_{\lambda}$ is the energy functional defined in \eqref{funz} and the positive constants $A$, $B$ are independent of $\lambda$ and $u$.

By considering the \emph{variational solutions} given by theorem \ref{mainth}, we now provide upper estimates of the critical values $c_k(\lambda)$ \eqref{minmaxlev}. We first remark that \emph{any} finite dimensional subspace $V$ with $\mathrm{dim}\,V>\mathrm{codim} V^+$ ($V^+$ being the subspace defined in \eqref{v2}) satisfies $V\in \Sigma_k$ for $k\le \mathrm{dim}\,V-\mathrm{codim} V^+$ (this is related to the so-called \emph{intersection lemma}, see \cite{struwe}, lemma $6.4$ and the proof of theorem 2.4 of \cite{bbf}).
Moreover, since for $\lambda\to 0^+$ we can take $V^+=H^1_a$, that is the subspace of codimension $1$ defined in \eqref{sub1},
we may assume $\mathrm{dim}\,V=k+1$.

Now, let us choose $k+1$ distinct points on the boundary $\partial\om$ and, for $\eps>0$, define
$$V_{\eps}\equiv \mathrm{span}\,\{-\log(\eps^2+|x-\sigma_j|^2),\quad j=1,2,...,k+1\}$$
By the estimates of lemmas $3.4$ and $3.5$ in \cite{vogekav} (with obvious modifications due to the expression \eqref{funz} of $E_{\lambda}$) one finds, for $\lambda$ sufficiently small and by taking $\eps$ of order $\lambda$ :
\begin{equation}
\nonumber
\max_{u\in V_{\eps}}E_{\lambda}(u)\le C_*\log(\frac{1}{\lambda})
\end{equation}
for some constant $C_*$ depending only on $k$, $a$ and $\om$. Since by the previous remark we have $V_{\eps}\in \Sigma_k$, the bound
\begin{equation}
\label{uppest}
c_k({\lambda})\le C_*\log(\frac{1}{\lambda})\
\end{equation}
follows.

Finally, the estimate \eqref{uppest} together with an elementary lemma in integration theory (lemma $3.6$ of \cite{vogekav})
yield the above mentioned bound on the normal current:
\begin{proposition}
Let $\lambda>0$ be small enough so that \eqref{uppest} holds. For $k\ge 1$, let $u_{k, \lambda}$ be the variational solutions obtained in theorem \ref{mainth} with $E_\lambda(u_k)=c_k({\lambda})$. Then, there exists a constant $D_*$ depending only on $k$, $a$ and $\om$, such that
\begin{equation}
\label{currest}
\int_{\partial\om}\Big |\frac{\partial u_{k, \lambda}}{\partial\nu} \Big |=
\lambda\int_{\partial\om}\big |\sinh(u_{k, \lambda})\big |\le D_*
\end{equation}
\end{proposition}
The proof follows by a trivial modification of the proof of Corollary $3.7$ in \cite{vogekav}.

By exploiting the previous estimates it is possible to describe the behaviour of a sequence of solutions
$u_{\lambda_n}= u_{k, \lambda_n}$, where $k\ge 1$ is fixed and $\lambda_n\to 0$. The results, that will be stated below,
can be proved by an adaptation of the arguments of section $4$ of reference \cite{vogekav}; the only non trivial
change is the use of a \emph{representation formula} for a classical solution $w$ of the Neumann problem
\begin{equation}
\label{neupb}
\Delta_a w=0,\quad \mathrm{in}\,\,\om\quad\quad \frac{\partial w}{\partial\nu}=f\quad \mathrm{on} \,\,\partial\om,
\quad \int_{\partial\om}a\,w=0
\end{equation}
(where $f$ is such that $\int_{\partial\om}a\,f=0$) which extends the usual layer potential representation for harmonic functions. We prove here this formula.
\begin{lemma}
Let $w$ be the solution to \eqref{neupb}. Then
\begin{equation}
\label{reprfor}
w(y)=\frac{1}{2\pi a(y)}\int_{\partial\om}a(\sigma)\,G_a(\sigma,y)\,f(\sigma)\,d\sigma
\end{equation}
where $G_a$ is the Green's function defined in \eqref{Greena}.\footnote{By an abuse of notation we write here $a(\sigma)$
(and similarly $G_a(\sigma,y)$, etc.) instead of $a(x)\big |_{\partial\om}$. }
\end{lemma}
\begin{proof}
We first remark that, by the boundary condition in \eqref{Greena} and by the normalization condition for $w$ in \eqref{neupb}
above, the following holds:
$$\int_{\partial\om}a(\sigma)\, w(\sigma)\,\partial_{\nu}G_a(\sigma,y)\,d\sigma=2\pi a(y) w(y)$$
Then, we can compute
$$
\int_{\partial\om}a(\sigma)\,G_a(\sigma,y)\,f(\sigma)\,d\sigma=\int_{\partial\om}a(\sigma)\,G_a(\sigma,y)\,\partial_{\nu} w(\sigma)\,d\sigma -\int_{\partial\om}a(\sigma)\, w(\sigma)\,\partial_{\nu}G_a(\sigma,y)\,d\sigma+2\pi a(y) w(y)
$$
$$
=\int_{\om}\,\mathrm{div}  \Bigl [G_a(x,y)\,a(x)\nabla_x w(x) -w(x)\,a(x)\nabla_x G_a(x,y)\Bigr ]\, dx+2\pi a(y) w(y)
$$
$$
=\int_{\om}  \Bigl [G_a(x,y)\,\mathrm{div}\big (a(x)\nabla_x w(x)\big)
-w(x)\,\mathrm{div}\big (a(x)\nabla_x G_a(x,y)\big )\Bigr ]\, dx+2\pi a(y) w(y)
$$
$$=2\pi a(y)\,w(y)$$
Hence, formula \eqref{reprfor} follows.
\end{proof}
As discussed above, one can now reproduce all the estimates proved in \cite{vogekav} (in particular, those in  Lemma $4.2$ of \cite{vogekav}) which lead to the following result
\begin{proposition}
\label{blowupvar}
Let $u_{\lambda_n}\in H^1(\om)$,  $\lambda_n\to 0^+$ be a sequence of solutions to \eqref{gen2dpb} given by theorem \ref{mainth}. Then, there exists a subsequence, also denoted by $u_{\lambda_n}$, a regular finite measure $m$ on $\partial\om$ and a finite set of points $\{x^{(i)}\}_{i=1}^N\subset\partial\om$, $N\ge 1$ such that
$$\lambda_n\big |\sinh(u_{\lambda_n}) \big |\rightarrow m$$
on $\partial\om$ in the sense of measures and the points $x^{(i)}$, $i=1,...,N$ are exactly the points at which $m$ has point masses. The same points also represent the blow up points for the sequence
$$u_{\lambda_n}^0=u_{\lambda_n}-\frac{1}{\int_{\partial\om}a}\,\int_{\partial\om}a\,u_{\lambda_n}$$
in the sense that
$$\{x^{(i)}\}_{i=1}^N=\bigl\{x\in\bar\om\,\,:\,\, \exists\,x_n\rightarrow x,\, x_n\in\bar\om,\,\,{\rm with}\,\,
|u_{\lambda_n}^0(x_n) | \rightarrow \infty \bigr\} $$
\end{proposition}

\section{Blowing up solutions}
\label{3}
We now show that there are solutions to problem \eqref{gen2dpb} which concentrate at isolated critical points $\xi_1$ and $\xi_2,$ say, of $a$ constrained on ${\partial\om}$ as $\lambda\to 0$. By the previous discussion, it corresponds to build up solutions to problem \eqref{nlbvp}
that concentrate positively and negatively along two geodesics of the boundary of $\mathcal D$ which are nothing but the $\mathcal G-$orbits of $\xi_1$ and $\xi_2$, respectively.

 This section is organized as follows. In subsection \ref{apprsol}, we write an approximate solution for problem \eqref{gen2dpb}.  In subsection  \ref{linprob}  we study an associate linear problem and in subsection \ref{nonlinprob} we reduce our non linear problem to a finite dimensional one. In subsection \ref{redprob} we study the reduced problem and we prove Theorem  \ref{principale}.

\subsection{The approximate solution}
\label{apprsol}

To define an approximate solution for the problem \eqref{gen2dpb}, a key ingredient is given by the solutions of the following problem
 \beq\label{apprpb}
 \left\{
 \begin{array}{lr}
\Delta v=0\qquad\mbox{in}\,\,\ \mathbb R^2_+\\\\
\displaystyle\frac{\partial v}{\partial\nu}= e^v\quad \mbox{on}\,\,\ \partial\mathbb R^2_+
\end{array}
\right.
\eeq
where $\mathbb R^2_+$ denotes the upper half plane $\left\{(x_1, x_2)\,:\, x_2>0\right\}$ and $\nu$ is the unit exterior normal to $\partial\mathbb R^2_+$.
\\ The solutions of \eqref{apprpb} are given by
 $$w_{t, \mu}(x_1, x_2)= \log \frac{2\mu}{(x_1-t)^2+(x_2+\mu)^2}$$
where $t\in\mathbb R$ and $\mu>0$ are parameters.\\\\
Let us   provide an approximation for the solution of our problem. Let

 $$
u_j^\lambda(x)=\log\frac{2\mu_j}{|x-\xi_j-\lambda\mu_j \nu_j|^2},\qquad j=1, 2\qquad \xi_j\in\partial \O,\,\, \mu_j>0.
 $$
In order to satisfy the equation $\Delta_a u=0$, we need an additional term  $H_j^\lambda$ defined as follows:
let $H_j^\lambda(x)$ to be the unique solution of
 $$
\left\{
\begin{array}{lr}
-\Delta_a H_j^\lambda=\nabla\log a\cdot \nabla u_j^\lambda,\qquad \mbox{in}\,\,\O\\\\
\displaystyle\frac{\partial H_j^\lambda}{\partial\nu}=-\displaystyle\frac{\partial u_j^\lambda}{\partial\nu}+\lambda e^{u_j^\lambda}-\lambda\frac{1}{\int_{\partial\O} a}\int_{\partial\O}a e^{u_j^\lambda}\quad \mbox{on}\,\, \partial\O\\\\
\displaystyle\int_{\partial\O}H_j^\lambda\, dx=-\displaystyle\int_{\partial\O}u_j^\lambda\, dx.
\end{array}
\right.
 $$

Now we set
 $$
U_\lambda(x)= \left[u_1^\lambda(x)+H_1^\lambda(x)\right]-\left[u_2^\lambda(x)+H_2^\lambda(x)\right]
 $$
and look for a solution of \eqref{gen2dpb} in the form
 $$
u_\lambda(x)= U_\lambda(x)+\Phi_\lambda(x)
 $$

The higher-order term   $\Phi_\lambda$ will satisfy some suitable orthogonality conditions (see \eqref{lineareprimo} below).
\\\\
Let $G_a(x, y)$ be defined as in \eqref{Greena} and $H_a(x, y)$ be the regular part of $G_a(x, y)$ defined as in \eqref{robina}. In the following we will write
simply $G(x, y)$ instead of $G_1(x,y)$ and $H(x, y)$ instead of $H_1(x,y).$

It is immediate to see that
$H_a(x, y)$ solves the following problem
 $$
\left\{
\begin{array}{lr}
(\Delta_a)_x H_a(x, y)=2\nabla\log a\cdot \displaystyle\frac{x-y}{|x-y|^2}\qquad \mbox{in}\,\ \O\\\\
\displaystyle\frac{\partial H_a}{\partial\nu_x}(x, y)=-\frac{2\pi}{\int_{\partial \O}a(x)}a(y)+\frac{2(x-y)\cdot\nu(x)}{|x-y|^2}\qquad \mbox{on}\,\, \partial\O
\end{array}
\right.
 $$

The function $H_j^\lambda$ can be estimated in terms of $H_a(x, y)$. Indeed, the following result holds.
\begin{lemma}\label{lemma:hjlambda}
For any $\alpha\in (0, 1)$
$$
H_j^\lambda(x)=H_a(x, \xi_j)-\log2\mu_j+O(\lambda^\alpha)
$$
uniformly in $\bar\O$.

\end{lemma}
\begin{proof}
The boundary condition satisfied by $H_j^\lambda$ is
\begin{eqnarray*}
\frac{\partial H_j^\lambda}{\partial\nu}&=&-\frac{\partial u_j^\lambda}{\partial\nu}+\lambda e^{u_j^\lambda}-\frac{\lambda}{\int_{\partial\O}a}\int_{\partial\O} ae^{u_j^\lambda}\\
&=& 2\lambda\mu_j\frac{1-\nu(\xi_j)\cdot\nu(x)}{|x-\xi_j-\lambda\mu_j\nu(\xi_j)|^2}+2\frac{(x-\xi_j)\cdot \nu(x)}{|x-\xi_j-\lambda\mu_j\nu(\xi_j)|^2}-\frac{\lambda}{\int_{\partial\O} a}\int_{\partial \O}ae^{u_j^\lambda}.
\end{eqnarray*}
As $\lambda\to 0$ we get:
\begin{eqnarray*}
\lambda\int_{\partial\O}a e^{u_j^\lambda}&=&\lambda\int_{\partial\O}a(x)\frac{2\mu_j}{|x-\xi_j-\lambda\mu_j\nu(\xi_j)|^2}= 2\int_{\frac{\partial\O-\xi_j}{\lambda\mu_j}}\frac{a(\xi_j+\lambda\mu_j y)}{|y-\nu(0)|^2}\\
&=& 2a(\xi_j) \int_{\frac{\partial\O-\xi_j}{\lambda\mu_j}}\frac{1}{|y-\nu(0)|^2}+O\left(\lambda |\nabla a(\xi_j)|\mu_j\int_{\frac{\partial\O-\xi_j}{\lambda\mu_j}}\frac{1}{|y-\nu(0)|^2}\right)\\
&=& 2a(\xi_j)\left(\int_{-\infty}^{+\infty}\frac{1}{1+t^2}\, dt -O\left(\int_{\lambda^{-1}\mu_j^{-1}}\frac{1}{1+t^2}\, dt\right)\right)\\
&=& a(\xi_j)\left(2\pi +O\left(\arctan(\lambda\mu_j)^{-1}-\frac \pi 2\right)\right)+O(\lambda\mu_j)\\
&=& 2\pi a(\xi_j)+O(\arctan(\lambda\mu_j))+O(\lambda\mu_j)\\
&=& 2\pi a(\xi_j)+O(\lambda\mu_j).
\end{eqnarray*}
Let us consider the difference
 $$
z_\lambda^{\xi_j} (x) = H_j^\lambda(x)-H_a(x, \xi_j) +\log 2\mu_j.
 $$
Hence $z_\lambda$ solves the following problem
 $$
\left\{
\begin{array}{lr}
-\Delta_a z_\lambda^{\xi_j} =-\Delta_a H_j^\lambda+\Delta_a H_a(x, \xi_j)\\\\
\displaystyle\frac{\partial z_\lambda^{\xi_j}}{\partial\nu}=\displaystyle\frac{\partial H_j^\lambda}{\partial\nu}-\displaystyle\frac{\partial H_a}{\partial\nu}
\end{array}
\right.
 $$
namely
 $$
\left\{
\begin{array}{lr}
-\Delta_a z_\lambda^{\xi_j} =\nabla\log a(x)\nabla\left[\log\frac{2\mu_j}{|x-\xi_j-\lambda\mu_j\nu_j|^2}-\log\frac{1}{|x-\xi_j|^2}\right]\\\\
\displaystyle\frac{\partial z_\lambda^{\xi_j}}{\partial\nu}=O\left(\lambda\mu_j\right)
\end{array}
\right.
 $$
As done in   Lemma 3.1 of \cite{davila2}  it follows that for any $p>1$ $$\|\frac{\partial z_\lambda^{\xi_j}}{\partial\nu}\|_{L^p(\partial\O)}\leq c\lambda^{\frac 1 p}.$$ Moreover, again as in Lemma 3.1 of \cite{davila2}, we get
$$\left\|\log\frac{1}{|x-\xi_j|^2}-\log\frac{1}{|x-\xi_j-\lambda\mu_j\nu(\xi_j)|^2}\right\|_{L^p(\O)}^p=\int_{B_{\lambda\mu_j}(\xi_j)\cap\O}\ldots +\int_{\O\setminus B_{\lambda\mu_j}(\xi_j)}\ldots =I_1+I_2$$
Now $$|I_1|\leq C\lambda^2\left(\log\frac 1\lambda\right)^p$$ while for $p\in (1, 2)$ $$|I_2|\leq C\lambda^p.$$ In conclusion, for any $p\in (1, 2)$ $$\left\|\log\frac{1}{|x-\xi_j|^2}-\log\frac{1}{|x-\xi_j-\lambda\mu_j\nu(\xi_j)|^2}\right\|_{L^p(\O)}\leq c\lambda.$$ By $L^p-$ theory $$\|z_\lambda^{\xi_j}\|_{W^{1+s, p}(\O)}\leq C\left(\left\|\frac{\partial z_\lambda^{\xi_j}}{\partial\nu}\right\|_{L^p(\partial\O)}+\|\Delta_a z_\lambda^{\xi_j}\|_{L^p(\O)}\right)\leq C\lambda^{\frac 1p}$$ for any $0<s<\frac 1p$. By   Morrey's embedding we obtain $$\|z_\lambda^{\xi_j}\|_{C^\gamma(\bar\O)}\leq c\lambda^{\frac 1p}$$ for any $0<\gamma<\frac 12+\frac 1p$. This proves the result with $\alpha=\frac 1p$.
\end{proof}
Moreover the function $H_a(x, y)$ can be expanded in terms of $H(x, y)$. The following expansion is proved in  Lemma 2.1 of \cite{wyz}.
\begin{lemma}\label{greena}
Let $H_{a, y}(x)=H_a(x, y)$ for any $y\in\O$. Then $y\to H_{a, y}$ is a continuous map from $\O$ into $C^{0, \gamma}(\bar\O)$, for any $\gamma\in(0,1)$. It follows
$$
H_a(x, y)=H(x, y)+\nabla \log a(y)\cdot \nabla(|x-y|^2\log|x-y|)+\mathcal{H}(x, y)
$$
where $y\to \mathcal{H}(\cdot, y)$ is a continuous map from $\O$ into $C^{1, \gamma}(\bar\O)$ for all $\gamma\in(0,1)$. Furthermore, the function $(x, y) \to \mathcal{H}(x, y)\in C^1(\O\times\O)$, in particular $x\to H_a(x, x)\in C^1(\O)$.
\end{lemma}

\smallskip
We consider now the following change of variables $$x=\lambda y,\qquad y\in\O_\lambda\equiv\frac{\O}{\lambda},\qquad v(y)= u(\lambda y).$$ Then $u$ is a solution to problem \eqref{gen2dpb} if and only if $v$ solves the problem
\beq\label{problemarisc}
\left\{
\begin{array}{lr}
\Delta_{a_\lambda} v=0\qquad\qquad \mbox{in}\,\, \O_\lambda\\\\
\displaystyle\frac{\partial v}{\partial \nu}=2\lambda^2 \sinh v \qquad \mbox{on}\,\, \partial\O_\lambda,
\end{array}
\right.
\eeq
where $a_\lambda(y):=a(\lambda y).$
In the expanded domain $\O_\lambda$, $U_{\lambda}(x)$ becomes
$$
V(y)=\sum_{j=1}^2(-1)^{j-1}\left[\underbrace{\log\frac{2\mu_j}{|y-\xi'_j-\mu_j\nu_j'|^2}-2\log\lambda}_{=u_j^\lambda(\lambda y)}+H_j^\lambda(\lambda y)\right]
$$
where $\xi_j'=\lambda^{-1}\xi_j$ and $\nu_j'=\nu(\xi_j')$. Therefore,
  \beq\label{vscaled}v(y)=V(y)+\phi (y),\qquad y\in\O_\lambda.\eeq
  will be a solution of \eqref{problemarisc} provided $\phi$ solves
$$
\left\{
\begin{array}{lr}
\Delta_{  a_\lambda} \phi=0\qquad\mbox{in}\,\, \O_\lambda\\\\
\displaystyle\frac{\partial\phi}{\partial\nu}- \mathcal W \phi= \mathcal R+ \mathcal N(\phi)\qquad \mbox{on}\,\, \partial\O_\lambda
\end{array}
\right.
$$
where we set
\beq\label{W}
\mathcal W(y): = 2\lambda^2\cosh V(y)
\eeq\beq\label{errore}
\mathcal R(y) :=-\left[\frac{\partial V}{\partial \nu}-2\lambda^2\sinh V\right](y)
\eeq
and
\beq\label{N}
\mathcal N(\phi)=2\lambda^2\left[\sinh (V+\phi)-\sinh V -(\cosh V)\phi\right].
\eeq

First of all, we prove that $V$ is a good approximation for a solution to  \eqref{problemarisc} provided the parameters $\mu_1$ and $\mu_2$ are suitably  choosen.

\begin{lemma}\label{errore-lem}
Assume
\beq\label{sceltamuj}
\log2\mu_1= H_a(\xi_1, \xi_1)-G_a(\xi_1, \xi_2)\ \hbox{and}\ \log2\mu_2= H_a(\xi_2, \xi_2)-G_a(\xi_2, \xi_1).
\eeq
 Then, for any $\alpha\in(0,1)$, there exists a positive constant $C$ independent of $\lambda$ such that, for any $y\in \O_\lambda$,
\beq\label{R}
|\mathcal R(y)|\leq \lambda^\alpha \sum_{j=1}^2\frac{1}{1+|y-\xi_j'|},\qquad \forall\,y\in \O_\lambda,
\eeq
and
\beq\label{W-esti}
\mathcal W(y)=\sum_{j=1}^2\frac{2\mu_j}{|y-\xi'_j-\mu_j\nu_j'|^2}(1+\theta_\lambda(y)),
\quad \hbox{with}\quad
|\theta_\lambda(y)|\leq C \lambda^\alpha + C\lambda \sum_{j=1}^2|y-\xi_j'|.
\eeq
\end{lemma}
\begin{proof}
By Lemma \ref{lemma:hjlambda} we deduce that if
  $|y-\xi_j'|\leq \frac{\delta}{\lambda}$
\begin{eqnarray*}
&& H_1^\lambda(\lambda y)-\left(\log\frac{2\mu_2}{\lambda^2|y-\xi_2'-\mu_2\nu(\xi_2')|^2}+H_2^\lambda(\lambda y)\right)\quad \hbox{(setting $z:=y-\xi'_1$)}\\
&&=(H_a(\lambda z+\xi_1,\xi_1)-\log2\mu_1)-\\
&&-\left(\log2\mu_2+\log\frac{1}{|\lambda z+(\xi_1-\xi_2)-\lambda\mu_2\nu(\xi_2')|^2}+H_a(\lambda z+\xi_1,\xi_2)-\log2\mu_2\right)+O(\lambda^\alpha)\\
&&= \underbrace{H_a(\xi_1,\xi_1)-\log2\mu_1 -G_a(\xi_1,\xi_2)}_{=0\ \hbox{because of}\ \eqref{sceltamuj}}+O(\lambda^\alpha)+O(\lambda|z|)\\
&&=O(\lambda^\alpha)+O(\lambda|y-\xi'_1|)
\end{eqnarray*}
and in a similar way
\begin{eqnarray*}
&&H_2^\lambda(\lambda y)-\left(\log\frac{2\mu_1}{\lambda^2|y-\xi_1'-\mu_1\nu(\xi_1')|^2}+H_1^\lambda(\lambda y)\right)=O(\lambda^\alpha)+O(\lambda|y-\xi'_2|).
\end{eqnarray*}
Therefore, the proof follows exactly as in  Lemma 3 of \cite{davila}.

\end{proof}

\subsection{A linear  problem}
\label{linprob}

The key ingredient in this section is
  the linearization of problem \eqref{apprpb} around the solution $w_{0, \mu}$, namely the problem

\beq\label{lineare}
\left\{
\begin{array}{lr}
\Delta\phi=0\qquad \qquad\mbox{in}\,\,\ \mathbb R^2_+\\\\
\displaystyle\frac{\partial\phi}{\partial\nu}=\displaystyle\frac{2\mu}{x_1^2+\mu^2}\phi\qquad \mbox{on}\,\,\partial\mathbb R^2_+.
\end{array}
\right.
\eeq

In \cite{davila} it has been proved the following result.
\begin{lemma}\label{boundedlin}
Any bounded solution of \eqref{lineare} is a linear combination of the functions
$$z_0(x)=1-2\mu\frac{x_2+\mu}{x_1^2+(x_2+\mu)^2}$$ and $$z_1(x)=-2\frac{x_1}{x_1^2+(x_2+\mu)^2}.$$
\end{lemma}

Now, let us assume that the points $\xi_1,\xi_2\in\partial\O$ are uniformly separated, namely  $|\xi_1-\xi_2|\geq d$  for some $d>0$ which does not depend on $\lambda.$
We have to redefine $z_0$ and $z_1$ in a neighbourhood of $\xi_1$ and $\xi_2$ in a suitable way. So, let $F_j: B_\rho(\xi_j)\to N_0$ a diffeomorphism, where $\rho>0$ is fixed and $N_0$ is an open neighborhood of $0\in\R^2$ such that $$F_j(\O\cap B_\rho(\xi_j))=\mathbb R^2_+\cap  N_0,\qquad F_j(\partial\O\cap B_\rho(\xi_j))=\partial\mathbb R^2_+\cap N_0$$ and such that $F_j$ preserves area.


For $y\in\O_\lambda\cap B_{\rho/\lambda}(\xi'_j)$ we define
$$
F^\lambda_j(y)=\frac{1}{\lambda}F_j(\lambda y)
\quad\hbox{and}\quad
Z_{ij}(y)=z_{ij}(F_j^\lambda(y)),\  j=1, 2,\  i=0, 1$$ where $z_{ij}$ denotes the function $z_i$ with parameter $\mu_j$, namely:
$$z_{0j}=1-2\mu_j\frac{x_2+\mu_j}{x_1^2+(x_2+\mu_j)^2},\qquad z_{1j}=-2\frac{x_1}{x_1^2+(x_2+\mu_j)^2}.$$
Let $\tilde\chi:\mathbb R\to\mathbb R$ be a non-negative smooth function with $\tilde\chi(r)=1$ for $r\leq R_0$ and $\tilde\chi(r)=0$ for $r\geq R_0+1$, $0\leq \tilde\chi\leq 1$ (with $R_0$ a large positive constant).
Then, we set
$$\chi_j(y):=\tilde\chi(|F_j^\lambda(y)|),\,\,j=1,2\quad\quad \mathrm{and}\quad\quad \chi(y):=\chi_1(y)+\chi_2(y)$$\

We will assume that $\lambda$ is small enough to satisfy
\beq\nonumber
|F^{\lambda}_j(y)|\ge R_0+1,\quad\quad \forall\,y\in\O_{\lambda}\cap \partial B_{\rho/{\lambda}}(\xi'_j)
\eeq
Hence, the products $\chi_j Z_{1j}$ can be defined in the whole domain $\O_{\lambda}$ by continuation by
zero in $\O_\lambda\backslash B_{\rho/\lambda}(\xi'_j)$.
Moreover, by the definition of $Z_{0j}$ we may also assume that, for fixed $0<b<1$ and suitable chosen $\delta$,
$$Z_{0j}(y)\ge 1-\lambda^b,\quad\quad \forall\,y\in\O_{\lambda}\cap \partial B_{\delta/{\lambda}}(\xi'_j) $$ \\

We now define:
\begin{equation}\label{Zy}
Z(y)=\left\{
\begin{array}{lr}
\min(1-\lambda^b, Z_{0j}(y))\qquad \mbox{if}\,\, |y-\xi_j'|<\frac{\delta}{\lambda}\\\\
1-\lambda^b\qquad\qquad\qquad \mbox{if}\,\, |y-\xi_j'|\geq \frac{\delta}{\lambda}\ \hbox{for}\ j=1,2.
\end{array}
\right.
\end{equation}
\\\\

We want to solve   the following linear problem: {\it given $f\in L^\infty(\Omega_\lambda)$ and $h\in L^\infty(\partial\O_\lambda)$, find $\phi\in L^\infty(\O_\lambda)$ and $c_j\in\mathbb R$, $j=0, 1, 2$ such that}
\begin{equation}\label{lineareprimo}
\left\{
\begin{array}{lr}
-\Delta_{a_\lambda  } \phi= f\qquad\qquad\mbox{in}\,\, \O_\lambda\\\\
\displaystyle\frac{\partial\phi}{\partial\nu}-\mathcal W\phi=h+\sum_{j=1}^2 c_j \chi_j Z_{1j}+c_0\chi Z\qquad \mbox{on}\,\, \partial\O_\lambda\\\\
\displaystyle\int_{\O_\lambda}a\chi Z\phi=0\ \hbox{and}\ \displaystyle\int_{\O_\lambda}a\chi_j Z_{1j}\phi=0\ \hbox{for}\ j=1, 2.
\end{array}
\right.
\end{equation}

It is necessary to introduce some $L^\infty$-weighted norms:
 if $h\in L^\infty(\partial\O_\lambda)$ and $f\in L^\infty(\O_\lambda)$, let $$\|h\|_*=\sup_{y\in\partial\O_\lambda}\frac{|h(y)|}{\sum_{j=1}^2(1+|y-\xi_j'|)^{-1-\sigma}}\quad\hbox{and}\quad \|f\|_{**}=\sup_{y\in\O_\lambda}\frac{|f(y)|}{\sum_{j=1}^2 (1+|y-\xi_j'|)^{-2-\sigma}}$$ where $\sigma>0$ is a fixed and small number.\\
The following result holds.

 \begin{proposition}\label{esistenzaphi}
For any  $d>0,$ there exist  $\lambda_0>0$ and $C>0$ such that for any $\lambda\in (0,\lambda_0),$  for any $\xi_1, \xi_2\in\partial\O$ with $|\xi_1-\xi_2|\geq d$, for any $h\in L^\infty(\partial\O_\lambda)$ and $f\in L^\infty(\O_\lambda)$ there is a unique solution $\phi\in L^\infty(\Omega_\lambda)$ and $c_0, c_1, c_2\in\mathbb R$ to the problem \eqref{lineareprimo}.\\ Moreover,
$$\|\phi\|_{L^\infty(\O_\lambda)}\leq C \log\frac{1}{\lambda}\left(\|h\|_{*}+\|f\|_{**}\right)\quad\hbox{and}\quad\max\{|c_0|, |c_1|, |c_2|\}\leq C \left(\|h\|_*+\|f\|_{**}\right).$$
\end{proposition}

\begin{proof}
We argue as in the proof of Proposition 1 and Proposition 2 of \cite{davila}. We only point out
$$\Delta_{a_\lambda}\phi(y)=\Delta\phi(y)+\lambda{\nabla a(\lambda y)\over a(\lambda y)}\phi(y),\ y\in{\Omega/\lambda}.$$
Moreover,   the proof exploits a potential theory argument where Green's function for the Laplacian is replaced by Green's function $G_a$ whose regular part is studied in Lemma \ref{greena}.

\end{proof}

\subsection{The non linear problem with constraints}
\label{nonlinprob}
In order to solve our problem we need to split the error term $\phi$ in \eqref{vscaled} as
 $\phi(y)=\tau Z(y)+\phi_1(y)$  where the function $Z$ is defined in \eqref{Zy},
$\tau=\tau(\lambda)$ is a small parameter and $\phi_1$ satisfies the orthogonal conditions
$$\displaystyle\int_{\O_\lambda}a\chi Z\phi_1=0\ \hbox{and}\ \displaystyle\int_{\O_\lambda}a\chi_j Z_{1j}\phi_1=0\ \hbox{for}\ j=1, 2.
$$
Therefore, the function $v$ in \eqref{vscaled} reads as
$$
v(y)=V_1(y)+\phi_1(y),\quad \hbox{where}\ V_1(y)=V(y)+\tau Z(y)\qquad y\in\O_\lambda.
$$
Moreover, $ v$ is a solution for \eqref{problemarisc} if and only if $ \phi_1$ solves
$$
\left\{
\begin{array}{lr}
-\Delta_{a_\lambda}  \phi_1=\tau \nabla \log a_\lambda\cdot\nabla Z\qquad \mbox{in}\,\,\ \O_\lambda\\\\
\displaystyle\frac{\partial  \phi_1}{\partial\nu}-\mathcal W_1  \phi_1= \mathcal R_1 +\mathcal N_1( \phi_1)\qquad \mbox{on}\,\,\partial\O_\lambda
\end{array}
\right.
$$
where (see also  \eqref{W}, \eqref{errore} and \eqref{N})
\beq\label{w1}
\mathcal W_1(y):=2\lambda^2 \cosh V_1(y)
\eeq
\beq\label{r1}
\mathcal R_1(y): =-\left[\frac{\partial V_1}{\partial\nu}-2\lambda^2 \sinh V_1\right](y)
\eeq
and
\beq\label{n1}
\mathcal N_1( \phi_1)=2\lambda^2\left[\sinh(V_1+ \phi_1)-\sinh V_1 -\cosh(V_1) \phi_1\right].
\eeq
It is important to point out that, since $Z(y)=O(1)$ on all $\Omega_\lambda,$ it follows that  $V_1(y)= V(y)+ O(|\tau|)$ for any $ y\in\O_\lambda. $
\\

Let us consider first the following auxiliary problem
\begin{equation}\label{nonlinaux}
\left\{
\begin{array}{lr}
-\Delta_{a_\lambda} \phi_1= \tau\nabla\log a_\lambda\cdot\nabla Z\qquad \qquad\qquad\qquad\qquad \mbox{in}\,\,\ \O_\lambda\\\\
\displaystyle\frac{\partial\phi_1}{\partial\nu}-\mathcal W_1 \phi_1= \mathcal R_1+\mathcal N_1(\phi_1) + c_0\chi Z+ c_1\chi_1 Z_{11}+c_2\chi_2 Z_{12}\qquad \mbox{on}\,\,\ \partial\O_\lambda\\\\
\displaystyle\int_{\O_\lambda}a\chi_j Z_{1j}\phi_1\, dx=0\  j=1,2, \quad \displaystyle\int_{\O_\lambda}a\chi Z\phi_1\, dx=0.
\end{array}
\right.
\end{equation}
where  $\mathcal W_1$, $\mathcal R_1$ and $\mathcal N_1$ are defined in \eqref{w1}, \eqref{r1} and \eqref{n1} respectively.\\
\begin{lemma}\label{lemma8}
Let $\alpha\in (0, 1)$, $d>0$   and $\tau=O(\lambda^\beta)$ with $\beta>\frac{\alpha}{2}$. Then there is $\lambda_0>0$ and $C>0$ such that for any $\lambda\in (0, \lambda_0) $ and for any $\xi_1, \xi_2\in\partial\O$ with $|\xi_1-\xi_2|\geq d$, problem \eqref{nonlinaux} has a unique solution $\phi_1\in L^\infty(\Omega_\lambda)$ and $ c_0, c_1, c_2\in\mathbb R$ such that
$$
\|\phi_1\|_{L^\infty(\O_\lambda)}\leq C \lambda^\alpha.
$$
Furthermore, the function $(\tau, \xi'_1,\xi'_2) \to \phi_1(\tau,\xi'_1,\xi'_2)\in L^\infty( \O_\lambda)$ is $ C^1$ and
$$
\|D_{ (\xi'_1,\xi'_2)}\phi_1\|_{L^\infty(\O_\lambda)}\leq C \lambda^\alpha\ \hbox{and}\ \|D_{\tau}\phi_1\|_{L^\infty(\O_\lambda)}\leq C \lambda^{\beta_1}\ \hbox{for some}\ \beta_1<\beta.
$$

\end{lemma}
\begin{proof}
We argue  as in the proof of Lemma 8 of \cite{davila}.  The only difference is due to the presence of the R.H.S. $f=\tau\nabla\log a_\lambda\cdot\nabla Z$ in  \eqref{nonlinaux}. \\ Indeed, first, we point out that
 $$\mathcal W_1(y)=\mathcal W(y)+\underbrace{2\lambda^2\sinh (V)\tau Z +\tau^2\lambda^2\cosh (V+\bar \tau Z)Z^2}_{:=\tau B},$$ where
 $\mathcal W$ is defined in \eqref{W} and $|\bar\tau|\leq |\tau|$. It is easy to check that  $\|B\|_*\leq C.$  Then we write the problem \eqref{nonlinaux} in terms of the operator $\mathcal A$ that associates to any $\phi_1\in L^\infty(\O_\lambda)$ the unique solution given by Proposition \ref{esistenzaphi} with $h=\tau B\phi_1+\mathcal R_1+\mathcal N_1(\phi_1)$ and $f=\tau\nabla\log a_\lambda\cdot\nabla Z$. In terms of $\mathcal A$, the problem \eqref{nonlinaux} is equivalent to the fixed point problem  $\phi_1=\mathcal A(\phi_1).$ Therefore, we are going to prove that $\mathcal A$ is a contraction mapping of   the set $$\mathcal C\equiv\left\{\phi\in C(\bar\O_\lambda)\,\,:\,\, \|\phi\|_{L^\infty(\O_\lambda)}\leq \lambda^\alpha\right\}.$$
From Proposition \ref{esistenzaphi} we get
$$\|\mathcal A(\phi_1)\|_{L^\infty(\O_\lambda)}\leq C |\log\lambda|\left[\underbrace{|\tau|\|B\phi_1\|_*+\|\mathcal N_1(\phi_1)\|_*+\|\mathcal R_1\|_*}_{:=\mathcal D}+\|f\|_{**}\right].$$
Arguing as in  \cite{davila} we get that $$\|\mathcal D\|_*\leq C\left(\lambda^{a-\sigma}+\lambda^{2\beta}\lambda^{\alpha+\beta}+\lambda^{2\alpha}\right)$$ for some $a\in(0,1)$ and $\sigma>0$ small so that $a-\sigma>\alpha$ ($\sigma$ is the number in the definition of $\|\cdot\|_*$, $\|\cdot\|_{**}$  and $\beta$  is such that $\tau=O(\lambda^\beta)$). On the other hand it is easy to check that
$$|f(y)|=O\left(\lambda \tau |\nabla Z(y)|\right)=O\left(\lambda^{1+a} \tau  \sum\limits_{j=1}^2 (1+|y-\xi_j'|)^{-1}\right)\ \hbox{for any}\ y\in \O_\lambda$$
and so
 $$\|f\|_{**}=O(\tau\lambda^{a-\sigma})=O(\lambda^{\beta+a-\sigma}).$$
 Then the proof  follows exactly as in   Lemma 8 of \cite{davila}.
\end{proof}

Next, we have to choose the parameter $\tau$ so that
 the nonlinear problem \eqref{nonlinaux} has a solution with $c_0=0.$ This is the result of next lemma whose proof can be carried out exactly as the proof of Lemma 9 in \cite{davila}.
\begin{lemma}\label{lemma9}
Let $d>0$. For any $\alpha\in (0, 1)$, there exist $\lambda_0>0$ and $C>0$ such that for $\lambda\in (0, \lambda_0)$, and any $\xi_1, \xi_2\in\partial\O$ with $|\xi_1-\xi_2|\geq d$, there exists a unique $\tau$ with $|\tau|< C\lambda^{\alpha-b/2}$ ($b$ is given in \eqref{Zy}), such that problem \eqref{nonlinaux} admits a unique solution $\phi_1\in L^\infty(\O_\lambda),$   $c_0=0$ and $ c_1, c_2\in\mathbb R.$ Moreover
\beq\label{stimaphi}
\|\phi\|_{L^\infty(\O_\lambda)}\leq C \lambda^\alpha
\eeq
and the function $(\xi'_1,\xi'_2) \to \phi_1(\xi'_1,\xi'_2)$ is $C^1$ and
\beq\label{stimederphi}
\|D_{ (\xi'_1,\xi'_2)}\phi_1\|_{L^\infty(\O_\lambda)}\leq C \lambda^\alpha.
\eeq
\end{lemma}
\subsection{The reduced problem and proof of Theorem \ref{principale}}
\label{redprob}

For any $ (\xi_1, \xi_2)\in\partial\O\times\partial\O$ with $
\xi_1\not=\xi_2$, we define $ \phi (\xi_1, \xi_2) $ and $ c_j
(\xi_1, \xi_2)$ for $j=1,2, $ to be the unique solution to
\eqref{nonlinaux} with $c_0=0$ satisfying \eqref{stimaphi} and
\eqref{stimederphi}.  In this section we shall find the points
$\xi_1$ and $\xi_2$ on the boundary $\partial\O$ such that
$c_1=c_2=0.$ That choice will provide a solution to  our
problem.\\\\

\begin{lemma}\label{cru}
Let $\alpha\in (\frac 12, \frac 34)$ and $b\in (2(1-\alpha), 1)$  ($b$ is given in \eqref{Zy}).
It holds true that
$$c_i=\tau\lambda\left[-{\mu_i\over2 a(\xi_i)} \nabla_{{\partial\Omega}}a(\xi_i)+o(1)\right],\ i=1,2$$
uniformly with respect to   $(\xi_1, \xi_2)\in
\partial\O\times\partial\O$ with $\xi_1\not=\xi_2.$
\end{lemma}
\begin{proof}
We multiply equation in \eqref{nonlinaux} by $a _\lambda
 \chi_jZ_{1j}$, $j=1, 2$ and we integrate in $y$. We take into account that $V_1=V+\tau Z$ and $\tau$ is choosen so that $c_0=0$ (see Lemma \ref{lemma9}).
Therefore we get
\begin{eqnarray*}
\int_{\O_\lambda}\tau\(\nabla a_\lambda\nabla Z\)\chi_jZ_{1j}dy&=&-\int_{\O_\lambda}a_\lambda \Delta_{a_\lambda}\(\chi_jZ_{1j}\)\phi_1dy\\
& &+\int_{\po}a_\lambda\phi_1\partial_\nu\(\chi_jZ_{1j}\)dy-\int_{\po}a_\lambda \chi_jZ_{1j}\partial_\nu\phi_1 dy
\end{eqnarray*}
and so
\begin{eqnarray}\label{RHS}
& &\sum\limits_{i=1}^2 c_i \underbrace{\int_{\po}a(\lambda y)\chi_i\chi_j Z_{1i}Z_{1j}dy}_{I_0} \\ & &=-\int_{\O_\lambda}a_\lambda \Delta_{a_\lambda}\(\chi_jZ_{1j}\)\phi_1dy-\int_{\O_\lambda}\tau\(\nabla a_\lambda\nabla Z\)\chi_jZ_{1j}dy\nonumber\\
& &+\int_{\po}a_\lambda\phi_1\partial_\nu\(\chi_jZ_{1j}\)dy-\int_{\po}a_\lambda \chi_jZ_{1j}\(\mathcal W_1\phi_1+\mathcal R_1+\mathcal N_1(\phi_1)\) dy\nonumber\\
&&=-\underbrace{\int_{\O_\lm} {a_\lm}\Delta_{a_\lambda}\(\chi_jZ_{1j}\)\phi_1dy}_{ I_1 }-\underbrace{ \int_{\O_\lambda}\tau\(\nabla a_\lambda\nabla Z\)\chi_jZ_{1j}dy}_{ I_2 }\nonumber\\ & &+
\underbrace{\int_{\po}a_\lambda\phi_1\left[\partial_\nu\(\chi_jZ_{1j}\) -\mathcal W_1\(\chi_jZ_{1j}\)\right]dy}_{I_3}\nonumber\\ &&
+\underbrace{\int_{\po}\left[\frac{\partial V}{\partial\nu}-\lm^2\sinh(V )\right]a_\lm \chi_j
Z_{1j}dy}_{I_4}+\underbrace{\tau\int_{\po}\left(\frac{\partial
Z}{\partial\nu}-\mathcal W Z\right)a_\lm\chi_j
Z_{1j}dy}_{I_5}\nonumber\\ & &-\underbrace{2\lm^2\int_{\po}a_\lm \chi_j
Z_{1j}\left[\sinh(V+\tau Z)-\sinh(V)-\tau \cosh(V)Z\right]dy}_{I_6}\nonumber\\ && -\underbrace{\int_{\po}\mathcal
N_1(\phi_1)a_\lm \chi_j Z_{1j}dy.}_{I_7}\nonumber
\end{eqnarray}
Now let us estimate each integral $I_i$'s of \eqref{RHS}.
It is immediate to check that
$$I_0={2\pi\over\mu_i} a(\xi_i)+o(1)\ \hbox{if}\ i=j\quad\hbox{and}\quad I_0=0\ \hbox{if}\ i\not=j.$$
Indeed if $i=j$
$$\int_{\po}a(\lambda y)\chi_i\chi_j Z_{1i}Z_{1j}dy =a(\xi_i)\int_{\mathbb R}{4x_1^2\over \(x_1^2+\mu_i^2\)^2}dx_1+o(1)=\frac{2\pi a(\xi_i)}{\mu_i}+o(1)$$
We remark that
$$Z_{1j}=O\left(\frac{1}{1+r}\right)\ \hbox{and}\ \nabla Z_{1j}=O\left(\frac{1}{1+r^2}\right)\ \hbox{as}\ r\to\infty .$$
We have
\begin{eqnarray*}
 {a_\lm}\Delta_{a_\lm}(\chi_j Z_{1j})&=&\Delta\chi_j \cdot Z_{1j}+2\n\chi_j\n Z_{1j}+\chi_j\Delta Z_{1j}+\n a_\lm \n \chi_j Z_{1j}+\chi_j \n a_\lm\n Z_{1j}\\
&=&O\left(\frac{\lm^2}{1+r}\right)+O\left(\frac{\lm}{1+r^2}\right)+\chi_j\Delta Z_{1j}+O\left(\frac{\lm^2}{1+r}\right)+O\left(\frac{\lm}{1+r^2}\right)
\end{eqnarray*}
However, $$\Delta_y Z_{1j}=\Delta_x z_1+O(\lm|x||\n^2z_1|)+O(\lm|\n z_1|)$$ and hence $$\Delta Z_{1j}=O\left(\frac{\lm}{1+r^2}\right)+O\left(\frac{\lm^2}{1+r}\right)$$
hence
$$|I_1|\leq C\lm\|\phi\|_{L^\infty(\O_\lm)}\leq C \lm^{1+\alpha}=o(\lm^{1+\alpha-\frac b2}).$$

Moreover,
\begin{eqnarray*}
|I_3|& &\le \left|\int_{\po}a_\lambda\phi_1 \partial_\nu\(\chi_j\)Z_{1j}dy\right|+\left| \int_{\po}a_\lambda\phi_1\chi_j
 \left(\partial_\nu Z_{1j}  -\mathcal W_1 Z_{1j} \right)dy\right|\\
 && \le C \lm \|\phi_1\|_{L^\infty(\O_\lm)}\log\frac 1\lm+C \lm^\alpha \|\phi_1\|_{L^\infty(\O_\lm)}\log\frac 1\lm\\
 & &\le C\(\lm^{1+\alpha}\log\frac
1\lm+\lm^{2\alpha}\log\frac
1\lm\)=o(\lm^{1+\alpha-\frac b2}) ,\end{eqnarray*}
 because $\n \chi_j=O(\lm)$ and (see also (3.39) in \cite{davila})
 $$\chi_j\( \partial_\nu Z_{1j} -\mathcal W_1
Z_{1j}\)=O\left(\frac{\lm^\alpha}{1+|y-\xi_j'|}\right)$$ and since $b>2(1-\alpha)$. \\
 By \eqref{R}  we get that
$$
 \mathcal R= (z^{\xi_1}_\lm(\lambda y)-z^{\xi_2}_\lm(\lm y))e^{H_a(\lm y, \xi_j)-H_a(\xi_j, \xi_j)}+O(\lm^2)
$$
 and hence by making the change of variable $x=F_j^\lm(y)$ and by observing that $(F_j^\lm)^{-1}(x)=x+\xi_j'+O(\lm|x|)$ we get
  \begin{eqnarray*}
I_4& = &2\mu_j\int_{-\frac{\rho}{\lm}}^{\frac \rho \lm}z^{\xi_j}_\lm(\lm x+\xi_j)a(\lm x+\xi_j)e^{\lm x+ O(\lm^2|x|^2)}\partial_{x_1}\frac{1}{x_1^2+\mu_j^2}\\
&=& 2\mu_j\int_{-\frac{\rho}{\lm}}^{\frac \rho \lm}\partial_{x_1}z^{\xi_j}_\lm(\lm x+\xi_j)a(\lm x+\xi_j)e^{\lm x+ O(\lm^2|x|^2)}\frac{1}{x_1^2+\mu_j^2}\\
&&-2\mu_j\lm \int_{-\frac{\rho}{\lm}}^{\frac \rho \lm}\partial_{x_1}a(\xi_j+\lm x+O(\lm^2|x|))z^{\xi_j}_\lm(\lm x+\xi_j)e^{\lm x+ O(\lm^2|x|^2)}\frac{1}{x_1^2+\mu_j^2}\\
&&-2\mu_j\lm \int_{-\frac{\rho}{\lm}}^{\frac \rho \lm}\partial_{x_1}z^{\xi_j}_\lm(\xi_j+\lm x+O(\lm^2|x|))a(\lm x+\xi_j)e^{\lm x+ O(\lm^2|x|^2)}\frac{1}{x_1^2+\mu_j^2}\\
&&-2\mu_j \lm \int_{-\frac{\rho}{\lm}}^{\frac \rho \lm}z^{\xi_j}_\lm(\xi_j+\lm x+O(\lm^2|x|))a(\lm x+\xi_j)\partial_{x_1}e^{\lm x+ O(\lm^2|x|^2)}\frac{1}{x_1^2+\mu_j^2}\end{eqnarray*}
and so
$$|I_4|\le \lm^{1+\alpha}\log\frac{1}{\lm} =o(\lm^{1+\alpha-\frac b2}).
$$
As done in the estimates  proved in p. 211, \cite{davila} we get that
$$\chi_j\( \partial_\nu Z_{} -\mathcal W
Z_{}\)= (z^{\xi_1}_\lm(\lambda y)-z^{\xi_2}_\lm(\lm y))e^{H_a(\lm y, \xi_j)-H_a(\xi_j, \xi_j)}+O(\lm^2)$$
and so, by making computations as before we get that
 \begin{eqnarray*}
|I_5|\le  C   \lm^{1+\alpha}\tau\log\frac 1\lm=o( \lm^{1+\alpha-\frac b2}).\end{eqnarray*}
Moreover, since by mean value theorem
$$\chi_j\lm^2\left[\sinh(V+\tau Z)-\sinh(V)-\tau \cosh(V)Z\right]=\chi_j\tau^2\sinh(V+\bar\tau Z)Z^2$$
by making again the same computations as before we get we also have
\begin{eqnarray*}
|I_6|\le  C   \lm\tau^2 \log\frac 1\lm =o(\lm^{1+\alpha- \frac b 2 }) .\end{eqnarray*}
Finally, we have
$$|I_7|\leq C \|\phi\|^2_{L^\infty(\O_\lm)}\log\frac 1\lm\leq C \lm^{2\alpha}\log\frac 1\lm.$$

Therefore, it remains to estimate the leading term $I_2$ of the R.H.S. of \eqref{RHS}.
We observe that in the function $Z$ is constant except in the regions $|y-\xi_j'|<\mu_j\lambda^{-b/2},$ $j=1,2.$
For sake of simplicity, we can also assume that   the boundary of $\O$ in a neighbourhood of the point $\xi_j$ can be described as a graph of a smooth function $\varphi_j$
defined in a neighbourhood of 0 such that $\varphi_j(0)=\varphi_j'(0)=0 $ so that
$$F_j(s_1,s_2)=\(s_1,s_2-\varphi_j(s_1)\)\ \hbox{and}\ F_j^{-1}(t_1,t_2)=\(t_1,t_2+\varphi_j(t_1)\)^.$$
Let us remind that $F_j^\lm(y)=\frac{F_j(\lm y)}\lm$, then
\begin{eqnarray*}
&&I_2= \int_{\O_\lambda}\tau\(\nabla a_\lambda(y)\nabla Z(y)\)\chi_j(y)Z_{1j}(y)dy\\ & &=\tau  \int_{\O_\lm}\nabla a(\lambda y)\nabla \(z_{0j}\left(  F_j^\lm (  y)\right)\)\tilde\chi \left( \left| F_j^\lm (  y)\right|\right)z_{1j}\left(  F_j^\lm (  y)\right)dy\\
 && =\tau\int_{\O_\lm}\left\{\frac{\partial a_\lm}{\partial y_1}\left[\frac{\partial z_{0j}}{\partial x_1} \frac{\partial(F_j^\lambda)_1}{\partial y_1} +\frac{\partial z_{0j}}{\partial x_2}\frac{\partial(F_j^\lambda)_2}{\partial y_1} \right]+\frac{\partial a_\lm}{\partial y_2}\left[\frac{\partial z_{0j}}{\partial x_1} \frac{\partial(F_j^\lambda)_1}{\partial y_2} +\frac{\partial z_{0j}}{\partial x_2} \frac{\partial(F_j^\lambda)_2}{\partial y_2} \right]\right\}\times\\ &&\hskip3truecm\times
\tilde\chi \left( \left| F_j^\lm (  y)\right|\right)z_{1j}\left(  F_j^\lm (  y)\right)dy\\
&&=\tau\lambda\int_{\O_\lm}\left\{\frac{\partial a}{\partial y_1}(\lm y)\left[\frac{\partial z_{0j}}{\partial x_1}\left(  F_j^\lm (  y)\right)-\frac{\partial z_{0j}}{\partial x_2}\left(  F_j^\lm (  y)\right)\varphi'_j(\lambda y_1) \right]+\frac{\partial a}{\partial y_2}(\lm y)\frac{\partial z_{0j}}{\partial x_2} \left(  F_j^\lm (  y)\right)\right\}\times\\ &&\hskip3truecm\times
\tilde\chi \left( \left| F_j^\lm (  y)\right|\right)z_{1j}\left(  F_j^\lm (  y)\right)dy\\
&&\hbox{(we set $ F_j^\lm (  y)=x$, i.e. $y=\frac{F_j^{-1}(\lm x)}\lm$)}\\
& &\\
&&=\tau\lambda\int_{B(0,\rho)\cap\mathbb R^2_+}\left\{\frac{\partial a}{\partial y_1}(\lm x_1,\lm x_2+\varphi_j(\lm x_1))\left[\frac{\partial z_{0j}}{\partial x_1}(x)-\frac{\partial z_{0j}}{\partial x_2}\left( x\right)\varphi'_j(\lambda x_1) \right]\right\}\times
\tilde\chi \left( \left|x\right|\right)z_{1j}\left(  x\right)dx\\
&&+\tau\lm\int_{B(0,\rho)\cap\mathbb R^2_+}\left\{ \frac{\partial a}{\partial y_2}(\lm x_1,\lm x_2+\varphi_j(\lm x_1))\frac{\partial z_{0j}}{\partial x_2} (x)\right\}\times\tilde\chi \left( \left|x\right|\right)z_{1j}\left(  x\right)dx\\
&&=\tau\lambda\left(\frac{\partial a}{\partial y_1}(0)\int_{\mathbb R^2_+}\frac{\partial z_{0j}}{\partial x_1}(x)z_{1j}\left(  x\right)dx+o(1)\right)=\tau\lambda\left(\frac{\partial a}{\partial y_1}(0)\int_{\mathbb R^2_+}\(-8\mu_j\frac{x_1^2(x_2+\mu_j)}{\(x_1^2+(x_2+\mu_j)^2\)^3}\right)dx+o(1)\right)\\
& &=\tau\lambda\(-\pi\nabla_{\partial\Omega} a(\xi_j)+o(1)\)
 \end{eqnarray*}
 We point out that this is the lower order term of the R.H.S. of \eqref{RHS}, because its rate is of order $\lambda^{1+\alpha-b/2}$ because of the  choice of $\tau.$ \\

 The claim follows collecting all the previous estimates.

\end{proof}
Finally, we conclude the proof of Theorem  \ref{principale}.\\

\begin{proof}[Theorems \ref{principale}  completed ]
Let $\xi_1$ and $ \xi_2$ be  two different $C^1-$stable critical points
 of $a $ restricted on $\partial\Omega,$  namely  the  local Brouwer degree
$$\mathrm{deg}\left(\nabla_{{\partial\Omega}}a ,  B(\xi_i,\rho) \cap \partial\Omega   ,0\right)\not=0\ \hbox{for}\ i=1,2$$  provided $\rho$ is small enough. Then by the product property we deduce
$$\mathrm{deg}\left(\left(\nabla_{{\partial\Omega}}a,\nabla_{{\partial\Omega}}a\right),  \left(B(\xi_1,\rho)\times  B(\xi_2,\rho)\right)\cap \left(\partial\Omega \times\partial\Omega\right) ,0\right) \not=0, $$ which implies together with
Lemma \eqref{cru}  that if
$\lambda$ is small enough there exists $(\xi_1^\lambda,\xi_2^\lambda)$ which approaches  $(\xi_1, \xi_2)$
as $\lambda$ goes to zero such that
$c_1=c_2=0.$ Therefore $v_\lambda(y)=V(y) + \phi(y)  $
$y\in\O_\lambda,$ turns out to be a solution to the  problem \eqref{problemarisc}. It is clear that
the scaled function $u_\lambda(x)=v_\lambda(x/\lambda),$ $x\in\O$ is
the solution to problem \eqref{gen2dpb}  which satisfies \eqref{solth} and  concentrates at the
points $\xi_1$ and $\xi_2$ as $\lambda$ goes to zero.
 \end{proof}


\bigskip

\begin{thebibliography}{9}
\bibitem{ambrab} A. Ambrosetti and P.H. Rabinowitz {\em
    Dual variational methods in critical point theory and applications},
    J. Funct. Anal. 14, 1973, 349-381.

\bibitem{bbf} P. Bartolo, V. Benci and D. Fortunato {\em
    Abstract critical point theorems and applications to some nonlinear problems with ''strong'' resonance at infinity},
    Nonlinear Analysis. Theory, Methods and Applications 7 (9), 1983, 981-1012.


\bibitem{BV} K. Bryan and M. Vogelius, {\em Singular solutions to a nonlinear elliptic boundary value problem
originating from corrosion modeling}, Quart. Appl. Math. 60, 2002, 675-694.

\bibitem{cherrier} P. Cherrier, {\em Probl\`{e}mes de Neumann nonlin\'{e}aires
    sur les vari\'{e}t\'{e}s riemanniennes},
    J. Funct. Anal. 57, 1984, 154-207.

\bibitem{davila2}
J. Davila, M. del Pino, M. Musso,  {\it Concentrating solutions in a two-dimensional elliptic problem with exponential Neumann data.}
J. Funct. Anal. {\bf 227} (2005), no. 2, 430--490.

\bibitem{davila}
J. Davila, M. del Pino, M. Musso, J. Wei, {\it Singular limits of a two dimensional boundary value porblem arising in corrosion modelling}, Arch. Rat. Mech. Anal. {\bf 182} (2006), 181--221





\bibitem{vogekav} O. Kavian and M. Vogelius {\em On the existence and
    'blow-up' of solutions to a two-dimensional nonlinear boundary-value problem arising in corrosion modeling},
    Proc. Roy. Soc. Edinburgh Sect A 133, 2003, 119-149. Corrigendum to the
    same, Proc. Roy. Soc. Edinburgh Sect A 133, 2003, 729-730.


\bibitem{PP1} C.D. Pagani and D. Pierotti {\em Variational methods for
    nonlinear Steklov eigenvalues problems with an indefinite weight function},
    Calc. Var. 39, 2010, 35-58.

\bibitem{PP2} C.D. Pagani and D. Pierotti {\em Multiple variational solutions
    to non linear Steklov problems}, Nonlinear Differ. Equ. Appl. DOI 10.1007/s00030-011-0136-z, 2011

\bibitem{PP3} C.D. Pagani and D. Pierotti {\em A three dimensional Steklov eigenvalue problem with exponential nonlinearity on the boundary},
Nonlinear Analysis TMA 79 28-40,  2013



\bibitem{ppvz} C.D. Pagani, D. Pierotti, G.M. Verzini and A. Zilio {\em
   A nonliner Steklov problem arising in corrosion modelling},
   Politecnico di Milano, QDD 189


\bibitem{rabino} P.H. Rabinowitz {\it Minimax methods in critical point theory
    with applications to differential equations.} CBMS Regional Conference
    Series in Mathematics, 65, AMS, Providence, RI, 1986.

    \bibitem{struwe} M. Struwe {\it Variational Methods and Applications to
    Nonlinear
    Partial Differential Equations and Hamiltonian systems},
    Springer Verlag, Berlin, 1990.


\bibitem{vogexiu} M. Vogelius and J.-M. Xu {\em A nonlinear elliptic boundary
    value problem related to corrosion modeling},
    Q. Appl. Math. 56, 1998, 479-505.


\bibitem{wyz}
J. Wei, D. Ye, F. Zhou, {\it Bubbling solutions for an anisotropic Emden-Fowler equation}, Calc. Var. PDE, {\bf 28} (2007), 217--247.


\end{thebibliography}
\end{document}